\documentclass[11pt]{amsart}
\usepackage{amsmath,amsthm,amsfonts,latexsym}

\newtheorem{theorem}{Theorem}[section]

\newtheorem{proposition}{Proposition}[section]

\newtheorem{remark}{Remark}

\def\eps{\varepsilon}
 \newcommand{\ue}{u^{\eps}} 
\newcommand{\we}{w^{\eps}}

 \newcommand{\Z}{{\mathbb Z}}

\newcommand{\R}{{\mathbb R}}

\begin{document}

\title[Speed of propagation for a semilinear parabolic equation]{Asymptotic speed of propagation for a viscous semilinear parabolic equation}

\author{Annalisa Cesaroni}
\address{Dipartimento di Scienze Statistiche
Universit\`a di Padova
Via Cesare Battisti 241/243,
35121 Padova, Italy}
\author{Nicolas Dirr}
\address{Cardiff School of Mathematics
Cardiff University,
Senghennydd Road, Cardiff, Wales, UK}
\author{Matteo Novaga}
\address{Dipartimento di Matematica
Universit\`a di Pisa,
Largo Bruno Pontecorvo 5,
56127 Pisa, Italy}

\begin{abstract} We characterize the asymptotic speed 
of propagation of almost planar solutions to a semilinear viscous parabolic equation, with periodic nonlinearity. 
\end{abstract}

%\begin{resume}
%Nous caract\'erisons la vitesse asymptotique
%de propagation des solutions presque planaires d'une \'equation parabolique 
%visqueuse semilin\'eaire, avec une non-lin\'earit\'e p\'eriodique.
%\end{resume}
%
%
\maketitle

\section*{Introduction}
We are interested in the asymptotic behavior, as $\eps\to 0$, of solutions to the following 
problem
\begin{equation}\label{uno} 
u^\eps_t - \eps^\alpha \Delta u^\eps +g\left(\frac{\ue}{\eps}\right)=0 \qquad x\in\R^n, t>0 
\end{equation} where $\alpha\in [0,1)$ and  $g:\R \to \R$ is a Lipschitz continuous, 
$1$-periodic function, with 
$g(v)\leq 0$.

If we perform the parabolic rescaling 
\[v^\eps(x,t)=\frac{1}{\eps} u^\eps(\eps x, \eps^{2-\alpha}t), \]
equation \eqref{uno} becomes 
\begin{equation}\label{ltb} v^\eps_t-  \Delta v^\eps +\eps^{1-\alpha} g\left(v^\eps  \right)=0. \end{equation}
Therefore, the analysis of the limit as $\eps\to 0$ of solutions to \eqref{uno} 
is related to the long time behavior of solutions of \eqref{ltb}.  
 
When $\eps=1$ or equivalently $\alpha=1$, the long time behavior has been   considered in the literature by several authors, also in the case  the laplacian is substituted by a more general elliptic operator, and  the existence of special solutions such as traveling or pulsating waves has been established (see \cite{dky, bh, nr, dy, bs, ls, cm, cls, cn} and references therein). 
%Using these results, the limit behavior of solutions to \eqref{uno} has been studied  in \cite{ls} (see also \cite{cls, cm, ac}). 
When $\int_0^1 g=0$, in \cite{j,m} it has been studied the limit behavior of solutions of the following  long time rescaling of \eqref{uno} \[u^\eps_t - \Delta u^\eps +\frac{1}{\eps}g\left(\frac{\ue}{\eps}\right)=0 \qquad x\in\R^n, t>0.\]

In this note we are interested in the case $\alpha\in [0,1)$.
The main result is Theorem \ref{omo} which provides the effective speed of propagation of solutions to \eqref{uno} starting from almost planar initial datum.  In particular we show that the solution  to \eqref{uno} with initial datum $u^\eps(x,0)=p\cdot x+k$ satisfies \[\lim_{\eps\to 0}u^\eps(x,t)=p\cdot x+k+c(p) t\quad \text{uniformly in $x,t$}\] where   the  average positive speed $c(p)$ depends on $g$. 

The proof of this result relies on maximum principle arguments and on a priori bounds on the gradient 
to solutions to \eqref{uno} with planar initial data. These bounds are provided by Theorem \ref{bernstein} and are based on the so called Bernstein method. 

 The interesting feature we observe is that in the case $\alpha<1$, the asymptotic speed of propagation  $c(p)$ is just lower semicontinuous, but not continuous with respect to $p$,  whereas when  $\alpha=1$, the speed $c(p)$  depends continuously on $p$, as it has been proved in \cite{dky, cls}.  In particular, we show  in Proposition \ref{effective} that \[c(p)=\begin{cases} - \int_0^1 g(s)ds & p\neq 0\\  -(\int_0^1 g^{-1}(s)ds)^{-1}& p=0\end{cases}\] so that the limit speed is {\em discontinuous} with respect to the slope $p$. Such phenomenon is unusual in homogenization problems, and  indicates that the effective limit of \eqref{uno} as $\eps\to 0$  is governed 
by a differential operator which is discontinuous in the gradient entry. An equation similar to \eqref{uno}, with $n=1$ and $\alpha=0$, giving rise to a discontinuous effective limit problem, has been  considered in \cite{cnv}.

This makes the analysis of this limit more challenging, 
and will be the topic of a  paper in preparation \cite{cddn}, where we will consider  
the long time behaviour and the asymptotic limit of solutions to 
\begin{equation}\label{unonuovo} 
u^\eps_t - \eps^\alpha \Delta u^\eps +g\left(\frac{x}{\eps},\frac{\ue}{\eps}\right)=0 \qquad x\in\R^n, t>0 
\end{equation} 
where  $g:\R^n\times\R\to \R$ is  Lipschitz continuous and $\Z^{n+1}$-periodic.  

\medskip

\noindent{\bf Acknowledgments.} The authors would like to thank 
the mathematics departments of the Universities of Padova, Firenze 
and Pisa for the kind hospitality during the preparation of this work.
A.C. was partially supported by the GNAMPA  Project 2015 ``Processi di diffusione degeneri o singolari legati al controllo di 
dinamiche stocastiche''. N.D was partially supported by the University of Padova through the Visiting Scientist Programme, and by the Leverhulme Trust through RPG-2013-261. M.N. was partially  supported  by  the University  of  Pisa  via Grant PRA-2015-0017.
 
\section{A priori estimates on the gradient}
We consider in this section the more general case in which $g:\R^n\times\R\to \R$ is a Lipschitz, $\Z^{n+1}$ periodic function.  

We introduce, for any $p\in\R^n$,  the initial value problem 
\begin{equation}\label{eq}
\begin{cases} v^\eps_t-  \Delta v^\eps +\eps^{1-\alpha} g\left(x, v^\eps  \right)=0 \\  v^\eps(x,0)=p\cdot x.\end{cases}
\end{equation}
We recall that this problem admits a unique solution $v^\eps\in C^{2+\gamma, 1+\gamma/2}$ for all $\gamma\in (0,1)$. Moreover we recall also that, due to Lipschitz regularity of $g$, a standard comparison principle among sub and supersolutions to \eqref{eq} holds (see \cite{pw}).   

We provide an apriori bound on the oscillation and on the gradient of the solutions to \eqref{eq}, which will be useful in the following. The approach is based on the so called Bernstein type method.

In the following we let $w^\eps(x,t)=v^\eps(x,t)-p\cdot x$, where $v^\eps$ is the unique  solution to \eqref{eq}. Notice that $w^\eps$ solves
\begin{equation}\label{eqw}
\begin{cases} w_t-  \Delta w +\eps^{1-\alpha} g\left(x, w+p\cdot x \right)=0 \\ w(x,0)=0.\end{cases}
\end{equation}

\begin{theorem}\label{bernstein} Given $p\in\R^n$, then there exists $\eps_0=\eps(|p|)$ such that for every $\eps\leq \eps_0$ and every $t\geq 0$ there holds 
\[ \sup_{x\in\R^n}|Dw^\eps(x,t)|\leq C_n \eps^{\frac{1-\alpha}{4}}(1+|p|)\|g\|_{1,\infty} \] 
and  \[\sup_{x\in \R^n} w^\eps(x,t)-\inf_{x\in \R^n}w^\eps(x,t)\leq 2+C_n \eps^{\frac{1-\alpha}{4}}(1+|p|)\|g\|_{1,\infty}\] where $C_n$ is a constant depending  on the space dimension and $\|g\|_{1,\infty}$ is the Lipschitz norm of $g$. 
\end{theorem} 
\begin{proof}
The proof of the theorem is based on similar arguments as in  \cite[Thm 2.4, Cor. 2.5]{dky}. We divide the proof in several steps. 

{\bf Step 1}: we prove that  for every $t\geq 0$ 
\begin{equation}\label{oscillation}{\rm osc}(w^\eps(\cdot, t),[0,1]^n)\leq {\rm osc}(w^\eps(\cdot, t),\R^n)\le {\rm osc}(w^\eps(\cdot, t),[0,1]^n)+2,\end{equation} where ${\rm osc}(w^\eps(\cdot, t), A)=\sup_{x\in A} w^\eps(x,t)-\inf_{x\in A}w^\eps(x,t)$. 

Fix $t> 0$ and $\ell \in \Z^n$ and let 
$\gamma= [p\cdot \ell] +1-p\cdot \ell \in (0, 1].$
We define  $\hat w^\eps(x,t):=w^\eps(x-\ell,t)+\gamma.$   Observe that, due to periodicity of $g$, it satisfies the same equation as $w^\eps$:
\[\hat w^\eps_t(x,t)= \Delta  w^\eps(x-\ell,t)-\eps^{1-\alpha} g(x-\ell, w^\eps(x-\ell, t)+p\cdot (x-\ell))= 
\Delta \hat w^\eps-\eps^{1-\alpha}g(x,\hat w^\eps+p\cdot x).\] Moreover $\hat w^\eps(x,0):=w^\eps( x-\ell,0)+\gamma=\gamma> 0$.
Therefore  by the comparison principle, we obtain  \[w^\eps(x,t)\leq w^\eps (  x-\ell,t)+ \gamma\leq w^\eps (  x-\ell,t)+ 1 \qquad \forall x\in\R^n, \ell\in\Z^n.\]  This implies immediately \eqref{oscillation}.  

{\bf Step 2}: properties of the function $W^\eps(t)=\sup_{x\in\R^n} w^\eps (x, t)$. 

 For simplicity from now on we consider the case in which $g\in C^\infty(\R^{n+1})$. The case of $g$ just Lipschitz is recovered by a standard approximation procedure. 

Note that, by  comparison principle, $0\leq w^\eps(x,t) \leq \eps^{1-\alpha}\|g\|_\infty t$, for all $x,t$.  Moreover, again by comparison principle,  
\[\inf_{y\in\R^n} w^\eps (y,s)\leq w^\eps (x,t)\leq  \sup_{ y\in\R^n} w^\eps (y,s)+\eps^{1-\alpha}(t-s)\|g\|_\infty \qquad\forall t\geq s.\]  
This implies that the function  \[W^\eps(t)=\sup_{x\in\R^n} w^\eps (x, t)\] satisfies, in viscosity sense,  \begin{equation}\label{weps} W^\eps_t\leq \eps^{1-\alpha} \|g\|_\infty .\end{equation}  

Let $\lambda>0$ to be fixed later. We consider the function 
\[\phi^\eps(x,t)=\frac{|D w^\eps(x,t)|^2}{2}-\lambda  w^\eps(x,t)+\lambda W^\eps(t).\] 
Note that $\frac{|D w^\eps(x,t)|^2}{2}\leq \phi^\eps(x,t)\leq \frac{|D w^\eps(x,t)|^2}{2}+\lambda{\rm osc}(w^\eps(\cdot, t),\R^n)$, and moreover $\phi^\eps(x,0)=0$. 
Let $\Phi^\eps(t)=\sup_{x\in\R^n} \phi^\eps(x,t)$.

{\bf Step 3}:  we prove that choosing $\lambda = C\eps^{(1-\alpha)/2} n^{1/2} (1+|p|) (\|Dg\|_\infty+\|g\|_\infty) $, for $C>1$, then for every $\eps$ sufficiently small  \[\Phi^\eps(t)\leq \lambda\sup_{t\in [0,T]} {\rm osc}(w^\eps(\cdot, t),\R^n)\] for every $t\in [0,T]$, where  $T>0$ is  the stopping time 
\[T= \inf \{s\geq 0:\  \sup_x |D w^\eps(x,s)|=1\}.\]

Let $t^\ast$ such that $\Phi^\eps(t^\ast)\geq \Phi^\eps(t)$ for every $t \in[0,T]$.  
If $t^\ast=0$, then   $\Phi^\eps(t)\leq 0$ for every $t$ and we are done. Assume that $t^\ast>0$. 
It is easy to show that there exists  a sequence $x_k(t^\ast)$ such that \[\phi^\eps(x_k(t^\ast),t^\ast)\to \Phi^\eps(t^\ast).\] 
We can choose the sequence such that   \[D\phi^\eps(x_k(t^\ast),t^\ast)\to 0\qquad \lim_k \Delta \phi^\eps(x_k(t^\ast),t^\ast)\leq 0.\] 
Finally,  \[\lim_k \phi^\eps_t(x_k(t^\ast),t^\ast)\geq 0.\]
 
We claim that, choosing appropriately $\lambda$, $\lim_k |Dw^\eps(x_k(t^\ast), t^\ast)|=0$. If it is true, then 
\[\Phi^\eps(t)\leq  \Phi^\eps(t^\ast)\leq  \lambda {\rm osc}(w^\eps(\cdot, t^\ast),\R^n)\leq\lambda\sup_{t\in [0,T]} {\rm osc}(w^\eps(\cdot, t),\R^n),\] and then we get the desired conclusion. 

Assume by contradiction that the claim is not true. Then, for every $\lambda$,  there exists a subsequence such that 
$\lim_k |Dw^\eps(x_k(t^\ast), t^\ast)|>0$.
We compute   
\begin{equation}\label{phi}  \begin{cases} \phi^\eps_{x_i} = \sum_j w^\eps_{x_i, x_j}\we_{x_j}-\lambda \we_{x_i}\\
\phi^\eps_t =\sum_j w^\eps_{ x_j, t}\we_{x_j}-\lambda \we_{t}+\lambda W^\eps_t\\
\phi^\eps_{x_i, x_i} = \sum_j w^\eps_{x_i, x_i, x_j}\we_{x_j}+\sum_j (w^\eps_{x_i,   x_j})^2-\lambda \we_{x_i, x_i}.\end{cases}
\end{equation}   
We recall that $w^\eps$ solves \begin{equation}\label{weq}w^\eps_t-  \Delta w^\eps +\eps^{1-\alpha} g\left(x, w^\eps+p\cdot x  \right)=0.\end{equation} Differentiating the previous equation with respect to $x_j$ we get \[w^\eps_{x_j,t}= \sum_i w^\eps_{x_i, x_i, x_j}-\eps^{1-\alpha} g_{x_j}\left(x, w^\eps+p\cdot x  \right)-\eps^{1-\alpha} g_{v}\left(x, w^\eps+p\cdot x  \right)(w^\eps_{x_j}+p_j).\]
So, using  this equation, \eqref{phi}, \eqref{weq} and \eqref{weps}, we get  \begin{eqnarray}\phi^\eps_t &=&\sum_j \sum_i w^\eps_{x_i, x_i, x_j}\we_{x_j}-\eps^{1-\alpha}\sum_j\we_{x_j}(g_{x_j}+p_j g_v)-\eps^{1-\alpha} \sum_j g_{v}(w^\eps_{x_j})^2-\lambda \we_{t} +\lambda W^\eps_t \nonumber
 \\ &=&  \Delta \phi^\eps -\sum_i\sum_j (w^\eps_{x_i,   x_j})^2+\lambda \Delta \we-\eps^{1-\alpha}\sum_j\we_{x_j}(g_{x_j}+p_j g_v)-\eps^{1-\alpha} \sum_j g_{v}(w^\eps_{x_j})^2-\lambda \we_{t}+\lambda W^\eps_t \nonumber \\
 &=&  \Delta \phi^\eps -\sum_i\sum_j (w^\eps_{x_i,   x_j})^2+\eps^{1-\alpha}[\lambda g-\sum_j\we_{x_j}(g_{x_j}+p_j g_v)- \sum_j g_{v}(w^\eps_{x_j})^2]+\lambda W^\eps_t \nonumber \\
 \nonumber &\leq &  \Delta \phi^\eps -\sum_i\sum_j (w^\eps_{x_i,   x_j})^2+2\eps^{1-\alpha} \lambda \|g\|_\infty + \eps^{1-\alpha} \|Dg\|_\infty (1+|p|)\left(\sum_j|\we_{x_j}|+\sum_j|\we_{x_j}|^2 \right)\\  \label{dise1} &\leq &  \Delta \phi^\eps -\sum_i\sum_j (w^\eps_{x_i,   x_j})^2+2\eps^{1-\alpha} \lambda \|g\|_\infty +\sqrt{n}\eps^{1-\alpha} \|Dg\|_\infty (1+|p|)(|D\we|+|D\we|^2).
\end{eqnarray}
We multiply the first equation in \eqref{phi} by $\we_{x_i}$ and compute at $x_k(t^\ast)$:  we have that \[\lim_k \sum_i\sum_j w^\eps_{x_i, x_j}(x_k(t^\ast))\we_{x_j}(x_k(t^\ast))\we_{x_i}(x_k(t^\ast))-\lambda |D\we (x_k(t^\ast))|^2= 0.\] By Cauchy-Schwartz inequality 
 \[\sum_i\sum_j w^\eps_{x_i, x_j}\we_{x_j}\we_{x_i}\le|D\we|\sum_i |\we_{x_i}| \left[\sum_j (w^\eps_{x_i, x_j})^2\right]^{1/2}\leq  |Dw^\eps|^2\left(\sum_i\sum_j (w^\eps_{x_i,   x_j})^2\right)^{\frac12}.\]
This implies, since $|D\we (x_k(t^\ast))|\neq 0$, that  \[\lim_k\sum_i\sum_j (w^\eps_{x_i,   x_j}(x_k(t^\ast)))^2 \ge  \lambda^2 .\] 
 We compute \eqref{dise1} at $x_k(t^\ast)$, recalling the definition of $T$, of the sequence $x_k(t^\ast)$ and \eqref{dise1} we get 
\begin{eqnarray} 0 & \leq &  \lim_k \left[\phi^\eps_t(x_k(t^\ast), t^\ast) -  \Delta \phi^\eps (x_k(t^\ast), t^\ast)\right]\nonumber \\   & \leq &  - 	\lambda^2 + 2\eps^{1-\alpha}
\lambda \|g\|_{ \infty}+\sqrt{n}\eps^{1-\alpha}\|Dg\|_{\infty} (1+|p|) \sup_{t\in [0,T]}\sup_x (|Dw^\eps|^2+ |Dw^\eps|) \nonumber \\ & \leq &  - 	\lambda^2 + 2\eps^{1-\alpha}
\lambda \|g\|_{ \infty}+\sqrt{n}\eps^{1-\alpha}\|Dg\|_{\infty} (1+|p|). \label{dis}
\end{eqnarray}

We claim it is possible to choose  $\lambda$  so that the right hand side of \eqref{dis} is strictly negative getting then a contradiction.  Indeed if we choose $\lambda = C\eps^{(1-\alpha)/2} n^{1/2} (1+|p|) (\|Dg\|_\infty+\|g\|_\infty) $, with $C>1$, the claim is true for every $\eps$ sufficiently small.

{\bf Step 4}: conclusion. 

By the first step, the definition of $\Phi^\eps$ and \eqref{oscillation}, we get that
\begin{eqnarray*}
\sup_x\frac{|D w^\eps(x,t)|^2}{2}\leq \Phi^\eps(t) &\leq &  \lambda\sup_{t\in [0,T]} {\rm osc}(w^\eps(\cdot, t),\R^n)\\  
&\leq&   \lambda\sup_{t\in [0,T]} ({\rm osc}(w^\eps(\cdot, t),[0,1]^n)+2) 
\\&\leq &\lambda\sup_{t\in [0,T]}\sup_x|D w^\eps(x,t)| +2\lambda\\
&\leq&  \sup_{t\in [0,T]}\sup_x\frac{|D w^\eps(x,t)|^2}{4} +\lambda^2+2\lambda\,, 
\end{eqnarray*}
which in turn gives 
\begin{equation}\label{stima}\sup_x|Dw^\eps(t,x) |^2 \le 4\lambda+2\lambda^2\qquad \forall t\in[0,T]. 
\end{equation}
Recall by the first step that $\lambda=  C\eps^{(1-\alpha)/2} n^{1/2} (1+|p|) (\|Dg\|_\infty+\|g\|_\infty)$, so choosing $\eps$ sufficiently small, we obtain  $4\lambda+2\lambda^2<1$. This implies, recalling \eqref{stima}, that the stopping time $T=+\infty$. 

Finally the desired estimate  on the gradient is obtained by \eqref{stima}, recalling the explicit formula for $\lambda$, and the estimate on the oscillation is deduced by recalling \eqref{oscillation}.  
\end{proof} 

\section{Existence of almost-planar solutions}
In this section, we show in fact that there exists for every $p\in\R^n$ a  positive constant $c_\eps(p)$ such that  solutions to \eqref{uno} starting from hyperplanes $z=p\cdot x$ remain  at a small distance from   hyperplanes with the same normal and  moving with (uniformly bounded in $\eps$) speed $c_\eps(p)\eps^{ \alpha-1}$.   Moreover we study the limit as $\eps\to 0 $ to $c_\eps(p)\eps^{ \alpha-1}$: this will be the average speed of hyperplanes. 
 \begin{theorem}\label{exc}
  Let $v^\eps$ be the solution to \eqref{eq}. 
	Then there exists a unique $c_\eps(p)>0$ such that \[|v^\eps(x,t)-p\cdot x-c_\eps(p) t|\leq K\] where $K$ depends 	on the Lipschitz norm of $g$.
Moreover $$0<c_\eps(p)\leq \eps^{1-\alpha}\|g\|_\infty\,.$$ 
\end{theorem}
The proof of this theorem is based on the apriori estimates provided in Theorem \ref{bernstein}, and follows as  in  \cite[Thm 3.1]{dky} (see also \cite{cddn} for a proof of a more general result of this kind). 

\begin{remark}\rm 
Observe that if $u^\eps$ is the  solution to \eqref{uno} with initial datum $u^\eps(x,0)=p\cdot x$, then 
$v^\eps(x,t)= \frac{1}{\eps} u^\eps(\eps x, \eps^{2-\alpha}t)$ is the  solution to  \eqref{eq}. So 
Theorem \ref{exc} implies
 \[|u^\eps(x,t)-p\cdot x-c^\eps(p)\eps^{\alpha-1} t|\leq K\eps.\] 
\end{remark} 
\begin{remark}\rm Note that, when $\alpha=1$, the speed $c^\eps(p)$ does not depend on $\eps$. 
When $\alpha=1$ and the laplacian in \eqref{uno} is substituted by 
the mean curvature operator, 
the existence of traveling wave solutions
has been established in \cite{dky}, under suitable assumptions on the forcing term $g$.  
\end{remark}

\smallskip

\begin{proposition}\label{effective} Let $c^\eps(p)$   as in Theorem \ref{exc}. Then
\begin{equation}\label{effc}
\lim_{\eps\to 0} c^\eps(p)\,\eps^{ \alpha-1}=c(p)=
\begin{cases}
-\int_0^1 g(s)ds & p\neq 0
\\ 
-\left(\int_0^1 \frac{1}{g(s)}ds\right)^{-1} & p=0,\  g<0
\\
0 & p=0,\ \max g=0 .\end{cases} \end{equation} 
\end{proposition}
 
\begin{proof}
For  $p=0$, equation \eqref{eq} becomes
\[
\begin{cases} v'(t)+\eps^{1-\alpha} g(v)=0
\\ v(0)=0.\end{cases}
\]
So, if $g<0$, by  direct integration we get 
\[\int_0^{v(t)} \frac{ds}{g(s)}=-\eps^{1-\alpha} t.\] 
In particular $v$ is monotone increasing, unbounded and
\[
\lim_{t\to +\infty} \frac{v(t)}{t} =-\eps^{1-\alpha} 
\left(\int_0^1 \frac{1}{g(s)}ds\right)^{-1} = c^\eps(0).
\]

If, on the other hand there exists $v_0\in[0,1]$ such that $g(v_0)=0$, then by comparison
$v_0-1\leq v(t)\leq v_0$ for every $t$, and then $c^\eps(0)=0$.  

For $p\neq 0$, we define $\chi_p:\R\times (0, +\infty)\to \R$ as $\chi_p(p\cdot x,t)=w^\eps(x,t)$, where $w^\eps$ solves \eqref{eqw}. 
Then $\chi_p$ solves
\begin{equation}\label{equnodim}
\begin{cases}(\chi_p)_t-|p|^2 (\chi_p)_{zz} +\eps^{1-\alpha}g\left(\chi_p(z,t)+z\right)=0 & t>0, z\in\R
\\ \chi_p(z,0)=0\, \end{cases}
\end{equation} and is $1$-periodic. We integrate in space and time \eqref{equnodim} and, using periodicity of $\chi_p$, we obtain  

% we integrate \eqref{equnodim}   in time and on a periodicity cell. 
\[   
\int_0^1\frac{1}{t}\int_0^t (\chi_p)_t dtdz=-\eps^{1-\alpha}\frac{1}{t}\int_0^t   \int_0^1  g(\chi_p+z)dzdt.
\] So \[ \frac{1}{t}\int_0^t (\chi_p)_t dt  =  \frac{\chi_p(z,t)}{t} \to c^\eps(p)\qquad t\to +\infty.\] 

Using  the a priori bounds on the  derivative $(\chi_p)_z$ given by Theorem \ref{bernstein} and the fact that $g$ is Lipschitz,  we get  \[\int_0^1  g(z+\chi(z,t))dz=\int_0^1 g\left(z+\int_0^1 \chi(\zeta,t)d\zeta \right)dz +O(\eps^{\frac{1-\alpha}{4}})=\int_0^1 g\left(z \right)dz +O( \eps^{\frac{1-\alpha}{4}})\]
where the error term $O(\eps^{\frac{1-\alpha}{4}})$ is independent of $t$.  
This gives \[c^\eps(p)= -\eps^{1-\alpha}  \int_0^1 g(z)dz +o(\eps^{1-\alpha}).  \]\end{proof} 

%
%If $g$ is depending on $x$ the computation of $c(p)$ is more involved.
%\begin{proposition} Let $g=g(x,v)$, $p$ with rational components and $c^\eps(p)$, $v^\eps_p$ as in Theorem \ref{exc}. Then
%\[\lim_{\eps\to 0} c^\eps(p)\eps^{ \alpha-1}=\begin{cases}... & p\neq 0\\ 
%-\left(\int_0^1 \frac{1}{\int_{(0,1)^n}g(x,s)dx}ds\right)^{-1} & p=0.\end{cases} \] \end{proposition}
%\begin{proof} We start with the case $p=0$. 
%In this case \eqref{eq} reads 
%\[\frac{\partial v^\eps }{\partial t}-  \Delta v^\eps +\eps^{1-\alpha} g\left(x,v^\eps  \right)=0. \]
%\end{proof} 

\section{Asymptotic speed of propagation}
In this section we provide the asymptotic speed of propagation of solutions to \eqref{uno} as $\eps\to 0$ starting from almost  planar initial datum.  When the initial datum  is an hyperplane, we obtain the homogenization limit of solutions to \eqref{eq}.
\begin{theorem}  \label{omo}
Let $u_\eps$ be the solution to \begin{equation}\label{eqeqplane}
\begin{cases}\frac{\partial u^\eps }{\partial t}- \eps^\alpha\Delta u^\eps +g\left( \frac{\ue}{\eps}\right)=0 & t>0, x\in\R^n
\\ u^\eps(x,0)= u_0(x).\end{cases}
\end{equation} 
Assume that there exists $p\in\R^n$ such that \begin{equation}\label{ap} u_0(x)-p\cdot x=v_0(x)\in L^\infty(\R^n).\end{equation}

Then    
\[  
p\cdot x +c(p)t +\inf v_0 \leq
\lim\inf_{\eps\to 0}  u^\eps(x,t)  
\leq \lim\sup_{\eps\to 0}   u^\eps(x,t)   
\leq p\cdot x +c(p)t+\sup v_0\, \] 
where $c(p)$ is defined in \eqref{effc}. 

In particular, if $u_0(x)-p\cdot x\equiv k$  for $k\in\R$, then the solution $u^\eps$ to \eqref{eqeqplane} satisfies 
\[\lim_{\eps\to 0} u_\eps(x,t)=p\cdot x + c(p)t +k\qquad \text{uniformly in $x\in\R^n, t\geq 0$}.\] 
\end{theorem} 

\begin{proof} 
 
We consider the solutions $u^\eps_+$, $u^\eps_{-}$ to \eqref{eqeqplane} with initial datum 
$\ue_+(x,0)=p\cdot x+\eps \left[\frac{\sup v_0}{\eps}\right]+\eps $ and $\ue_{-}(x,0)=p\cdot x+\eps \left[\frac{\inf v_0}{\eps}\right] $. By comparison $\ue_{-}(x
,t)\leq \ue(x,t)\leq \ue_+(x,t)$ for every $x,t$. 
We can rewrite 
\[
\ue_{-}(x,t)=p\cdot x +\eps w^\eps\left(\frac{x}{\eps}, \frac{t}{\eps^{2-\alpha}}\right)+\eps \left[\frac{\inf v_0}{\eps}\right]\] and analogously $\ue_+$, 
where $w^\eps$ is the solution to \eqref{eqw}. 

%By uniqueness we get    that $v_\eps(x,t)=\chi_p(p\cdot x, t)$, where $\chi_p(z,t)$ is $1$ periodic in $z$ and solves 
%\begin{equation}\label{equnod}
%\begin{cases}(\chi_p)_t-|p|^2 (\chi_p)_{zz} +\eps^{1-\alpha}g\left(\chi_p(z,t)+z\right)=0 & t>0, z\in\R
%\\ \chi_p(z,0)=0.\end{cases}
%\end{equation}
%By theorem \ref{exc}, we get that there exists $C$ independent of $\eps$ such that \begin{equation}\label{lim} |\chi_p(z,t)- c^\eps(p)t|\leq C.\end{equation} 

So we get that \begin{equation}\ue_{-}(x,t)=p\cdot x+\eps^{\alpha-1}c^\eps(p)t +
\eps \left(w^\eps\left(\frac{x}{\eps}, \frac{t}{\eps^{2-\alpha}}\right)-c^\eps(p)\frac{t}{\eps^{2-\alpha}}\right)+\eps \left[\frac{\inf v_0}{\eps}\right]  \end{equation} and analogously for $\ue_+$. 
By comparison we get that 
\begin{eqnarray*}
\eps \left(w^\eps\left(\frac{  x}{\eps}, \frac{t}{\eps^{2-\alpha}}\right)-c^\eps(p)\frac{t}{\eps^{2-\alpha}}\right)+\eps \left[\frac{\inf v_0}{\eps}\right] 
&\leq& u^\eps(x,t)-p\cdot x-\eps^{\alpha-1}c^\eps(p)t 
\\
&\leq& 
\eps \left(w^\eps\left(\frac{  x}{\eps}, \frac{t}{\eps^{2-\alpha}}\right)-c^\eps(p)\frac{t}{\eps^{2-\alpha}}\right)+\eps  \left[\frac{\sup v_0}{\eps}\right] +\eps.
\end{eqnarray*}
Recalling Theorem \ref{exc}, and  letting $\eps\to 0$,  we get the thesis.

%\[\inf v_0\leq \lim\inf_{\eps\to 0} u^\eps(x,t)-p\cdot x-\lim\sup_{\eps\to 0}\eps^{\alpha-1}c^\eps(p)t\leq \sup v_0\] 
%and 
%\[\inf v_0\leq \lim\sup_{\eps\to 0} u^\eps(x,t)-p\cdot x-\lim\inf_{\eps\to 0}\eps^{\alpha-1}c^\eps(p)t\leq \sup v_0.\] 
%When $p=0$, then $v_\eps(x,t)=\chi_0(t)$, where $\chi_0$ solves \begin{equation}\label{equnod00}
%\begin{cases} \chi_0'  +\eps^{1-\alpha} g\left(\chi_0\right)=0 & t>0, 
%\\ \chi_0(0)=0.\end{cases} 
%\end{equation}  
%We repeat the same argument as above and we find, using comparison principle
%\[\eps\chi_0\left(\frac{t}{\eps^{2-\alpha}}\right)+\eps\left[\frac{\inf v_0}{\eps}\right]\leq u^\eps(x,t)\leq \eps\chi_0\left(\frac{t}{\eps^{2-\alpha}}\right)+\eps \left[\frac{\sup v_0}{\eps}\right] +\eps\]
%and we conclude by observing that 
%\[\chi_0(t)-\eps^{\alpha-1}c^\eps(0)t \] is bounded. 
\end{proof}

%%%%%%%%%%%%%%%%%%%%%%%%%%%%%%%%%%%%%%%%%%%%%%%%%%%%%%%%%%%%%%%%%%%%%%%%%%%%%%


\begin{thebibliography}{99}
\bibitem{m} N. Alibaud, A. Briani, R. Monneau. 
\newblock Diffusion as a singular homogenization of the Frenkel-Kontorova model. 
\newblock {\em J. Differential Equations } 251, no. 4-5, 785--815, 2011. 
\bibitem{bs} G. Barles, P. E. Souganidis.
\newblock Space-time periodic solutions 
and long-time behavior of solutions to quasi-linear parabolic equations. 
\newblock{\em SIAM J. Math. Anal. }32, no. 6, 1311--1323, 2011.
\bibitem{bh}
H. Berestycki, F. Hamel.
\newblock Generalized travelling waves for reaction-diffusion equations.
\newblock In: {\em Perspectives in Nonlinear Partial Differential Equations. In honor of H. Brezis.} 
Contemp. Math. 446, Amer. Math. Soc., 101--123, 2007.
\bibitem{cm} L.A. Caffarelli, R. Monneau.
\newblock Counter-example in three dimension and homogenization of geometric motions in two dimension.
\newblock{\em Arch. Ration. Mech. Anal.} 212, no. 2, 503--574, 2014.
\bibitem{cls} P. Cardaliaguet, P.-L. Lions, P. E. Souganidis. 
\newblock A discussion about the homogenization of moving interfaces.
\newblock{\em J. Math. Pures Appl.} 91, 339--363, 2009.
\bibitem{cddn} A. Cesaroni, N. Dirr, F. Dragoni,  M. Novaga.  
\newblock In preparation.
\bibitem{cn} A. Cesaroni, M. Novaga.  \newblock Long-time behavior of
  the mean curvature flow with periodic forcing.  
\newblock {\em Comm. Partial Differential Equations} 38, 780--801, 2013.
\bibitem{cnv} A. Cesaroni, M. Novaga, E. Valdinoci.  
\newblock Curve shortening flow in heterogeneous media.
\newblock {\em Interfaces and Free Boundaries} 13(4), 617--635, 2011.
\bibitem{dky} N. Dirr, G. Karali, N. K. Yip.
\newblock Pulsating wave for mean curvature flow in inhomogeneous medium. 
\newblock {\em Eur. J. of Applied Mathematics} 19, 661--699, 2008.
\bibitem{dy} N. Dirr, N. K. Yip. 
\newblock Pinning and de-pinning phenomena in front propagation in heterogeneous media. 
\newblock{\em Interfaces and Free Boundaries} 8(1), 79--109, 2006.
\bibitem{ds} N. Dirr, P.E. Souganidis.
\newblock Large-time behavior for viscous and nonviscous Hamilton-Jacobi equations forced by additive noise. \newblock{\em SIAM J. Math. Anal.} 37, no. 3, 777�-796, 2005.  
\bibitem{j} R. Jerrard, 
\newblock Singular limits of scalar Ginzburg-Landau equations with multiple-well potentials, 
\newblock {\em Adv. Differential Equations} 2 (1)  1--38, 1997.
\bibitem{ls} P.-L. Lions,  P. E. Souganidis. 
\newblock Homogenization of degenerate second-order PDE in periodic and almost periodic environments and applications. 
\newblock{\em Ann. Inst. H. Poincar\'e Anal. Non Lin\'eaire} 22, no. 5, 667�-677, 2005. 
\bibitem{nr} G. Namah, J.-M. Roquejoffre.
\newblock Convergence to periodic fronts in a class of semi-linear parabolic equations.
\newblock {\em Nonlinear Diff. Eq. Appl.} 4, 521�-536,1998.
\bibitem{pw} M.~H. Protter, H.~F. Weinberger.
\newblock {\em Maximum principles in differential equations}. 
\newblock Springer-Verlag, New York, 1984.
\end{thebibliography}
\end{document}